\documentclass[10pt,a4paper]{amsart}

\usepackage[left=3.3cm,right=3.3cm,bottom=4cm,top=4.65cm]{geometry}

\usepackage{amsmath,amsthm,amssymb}

\usepackage[pagebackref=true]{hyperref}
\hypersetup{
	colorlinks = true,
	pdfstartview = FitH,
	pdfpagemode = UseNone
}

\usepackage{enumerate}
\usepackage{multicol}
\usepackage{array}
\usepackage{caption}

\numberwithin{equation}{section}
\numberwithin{equation}{section}

\theoremstyle{plain}
\newtheorem{theorem}{Theorem}[section]
\newtheorem{lemma}[theorem]{Lemma}
\newtheorem{prop}[theorem]{Proposition}
\newtheorem{cor}[theorem]{Corollary}

\title{There are Salem numbers
with trace $-3$ and every degree at least $34$}

\author{Giacomo Cherubini}

\address{
	Faculty of Mathematics and Physics, Department of Algebra,
	Charles University,
	Sokolov\-sk\' a 83, 18600 Praha~8,
	Czech Republic.
}
\address{
	Istituto Nazionale di Alta Matematica ``Francesco Severi'',
	Research Unit Dipartimento di Matematica ``Guido Castelnuovo'',
	Sapienza Universit\`a di Roma, Piazzale Aldo Moro 5, I-00185, Roma.
	\vspace{5pt}
}

\email{
	cherubini@karlin.mff.cuni.cz,
	cherubini@altamatematica.it
}

\author{Pavlo Yatsyna}

\address{
	Faculty of Mathematics and Physics, Department of Algebra,
	Charles University,
	Sokolov\-sk\' a 83, 18600 Praha~8,
	Czech Republic.
}
\address{
	Department of Mathematics and Systems Analysis,
	Aalto University,
	P.O. Box 11100, FI-00076,
	Finland.
	\vspace{5pt}
}

\email{
	p.yatsyna@matfyz.cuni.cz
}

\date{\today}

\begin{document}

	\begin{abstract}
		We prove that there exist Salem numbers with trace $-3$
		and every even degree $\geq 34$.
		Our proof combines a theoretical approach,
		which allows us to treat all sufficiently large degrees,
		with a numerical search for small degrees.
		Since it is known that there are no Salem numbers of trace $-3$
		and degree $\leq 30$, our result is optimal up to possibly the single
		value $32$, for which it is expected there are no such numbers.
	\end{abstract}

	\maketitle\thispagestyle{empty}

	\section{Introduction}

	A Salem number is a real algebraic integer greater than one
	with all its conjugates in the closed unit disc,
	and at least one of them on the unit circle.
This implies that a Salem number must have an even degree of at least $4$, it is conjugate to its inverse, and all its other conjugates must have absolute value one.

	There is abundant literature on Salem numbers; we refer to the excellent survey by Smyth
	\cite{S2} for an overview of results and applications on this topic.

	A simple example of a Salem number is the largest real root of
	\begin{equation}\label{0712:eq001}
	x^4 - nx^3 - (2n+1)x^2 - nx +1,
	\end{equation}
	for any natural number $n$ \cite[p.~35]{MS3}.
	This family shows that Salem numbers can have arbitrary positive trace
	(there is an example with trace zero and degree $6$, too).
	
	It is much more difficult to find Salem numbers with \emph{negative} trace. Indeed, it was not until 2005 that McKee and Smyth \cite{MS3} showed that such numbers exist for every trace. However, there is no family of a fixed degree giving all negative traces, such as the one described by \eqref{0712:eq001}. Instead, the degree has to increase as the trace decreases. For instance, when the trace is $-1$, the degree must be at least $8$, which is attained, as is in every even degree larger than that \cite{S1}.

	When the trace is $-2$, we need to jump to degree $20$
	to find a Salem number \cite{MS1}; then in degree $22$ there is none
	and after that every even degree $\geq 24$ is again attained~\cite{MY},
    see also \cite{CW}.
	Our main result concerns the existence of Salem numbers of trace $-3$.

	\begin{theorem}\label{intro:thm}
		For each even degree $d\ge 34$ there is Salem number of trace $-3$ and degree $d$.
	\end{theorem}

	In \cite[Corollary~2]{WWW}, Wang, Wu and Wu proved that there is
	no Salem number of trace $-3$ and degree $\leq 30$
	(this took more than six months of computations \cite[p.~2319]{WWW}).
	Therefore, Theorem \ref{intro:thm} is sharp up to possibly
	the single value $32$ missing from our classification.
	It has been suggested that Salem numbers of trace $-3$
	do not exist with such a degree \cite{EMRS}.
	
	In any event, we observe that, unlike the case where the trace $-2$, there is no gap in the list of even degrees for which Salem numbers of trace $-3$ exist, which is similar to the situation with $-1$. This raises the interesting question of whether this behaviour will repeat in the next few cases. In view of the work in \cite{WWW}, it is likely that even settling trace $-4$ may be a computationally challenging task. 
	
	We further point out that McKee in \cite[\S5.4]{M} found Salem numbers of trace $-3$ and degree $54$ (as well as one with degree $76$), and hinted that smaller degrees could be found. Our Theorem \ref{intro:thm} confirms this intuition, with the smallest known degree of Salem numbers of trace $-3$ being due to S.~El~Otmani, G.~Rhin, and J.-M. Sac-\'{E}p\'{e}e \cite{ERS}. For part of our proof, we apply McKee's method, as described in Section~\ref{S2.4}.

	Finally, we have a direct application of Theorem \ref{intro:thm} to totally positive
	algebraic integers (i.e., such that all their conjugates are positive).
	
	\begin{cor}\label{intro:cor}
		Totally positive algebraic integers with degree $d$ and trace $2d-3$ exist for every~$d\geq 16$.
	\end{cor}
	
	It is well known that if $P(x)$ is the minimal polynomial of a Salem number (a Salem polynomial, so to speak), then there exists an irreducible polynomial $Q \in \mathbb{Z}[x]$, of degree half the degree of $P$, such that $P(x)=x^{\deg Q}Q(x+\frac{1}{x}+2)$. Moreover, $Q$ has only real and positive roots. Consequently, Theorem \ref{intro:thm} directly implies Corollary \ref{intro:cor} for $d>16$, and a polynomial of degree $16$ and trace $29$ was found in \cite{EMRS}. Furthermore, Corollary \ref{intro:cor} is complete, since there is no totally positive algebraic integer with trace $2d-3$ and degree $<16$.
	
	\subsection*{Acknowledgements}
    G.C.~was supported by Czech Science Foundation GACR, grant 21-00420M and is currently a researcher at INdAM. 
	P.Y.~was supported by Charles University, projects PRIMUS/24/SCI/010 and UNCE/24/SCI/022, the Academy of Finland (grants 336005 and 351271, PI C. Hollanti), and by MATINE, Ministry of Defence of Finland (grant 2500M-0147, PI C. Hollanti).

	\section{Proof}

	Our proof of Theorem \ref{intro:thm} develops in four steps,
	each dealing with the existence of Salem numbers for different degrees.

	\emph{Step 1.}
	We start by constructing suitable infinite families of polynomials
	associated to seven-tuples of primes.
	A finite check on the initial choice of primes ensures that
	the polynomials in the family are irreducible and Salem.
	Our construction is similar to the one in \cite{MY}
	and makes use of interlacing polynomials.
	Furthermore, by selecting sufficiently many tuples,
	we can show that we obtain Salem polynomials 
	of trace $-3$ and every degree strictly greater than $80$
	(degrees $72$, $76$ and $78$ are also obtained).
	This is proved in Proposition \ref{prop:tuples}.

	\emph{Step 2.}
	For Salem numbers of degree $\leq 80$ (and different from $72, 76, 78$)
	we need to use different constructions.
	Our second step consists in adapting the construction of Step~1,
	using quadruples of primes, again interlacing polynomials
	and an auxiliary pair of degree-eight polynomials.
	This allows us to deal with the degrees $54,56,60,62,\dots,70,74$ and $80$,
	as demonstrated in Lemma \ref{2007:lemma1}.

	\emph{Step 3.}
	In degree $34,\dots, 40$, Salem polynomials have been shown to exist
	in \cite{ERS} using optimisation techniques, while in degree $32$ it is expected
	that there are no Salem numbers of trace $-3$ (see \cite[p.~22]{ERS}).
	This is reinforced by the numerical search in \cite{EMRS},
	where a similar method produced eleven totally positive polynomials
	of degree 16 and trace 29, none of which was Salem.

	\emph{Step 4.}
	Finally, for the degrees $42,\dots, 52$ and $58$, we do an extensive computer search
	(which took almost two weeks)
	applying a version of the method described in Sections~2.1 and~2.2 of~\cite{M}.
	The goal here is to construct monic integer polynomials that have only real roots
	and are in one-to-one correspondence with Salem polynomials (see details below).

	\subsection{Interlacing polynomials and seven-tuples of primes}\label{S2.1}

	We use the method from \cite{MS3,MS2} (specifically, Section~3 in~\cite{MS3})
	to construct Salem numbers.
	This is based on the use of cyclotomic polynomials and the notion of interlacing,
	which allows us to easily obtain Salem polynomials.
	
	We define a cyclotomic polynomial as any monic integer polynomial with all roots on the unit circle. In particular, this includes polynomials with multiple irreducible factors. We will use the notation $\Phi_n$ to denote the minimal polynomial of an $n$th root of unity. Given two coprime cyclotomic polynomials, $P$ and $Q\in\mathbb{Z}[x]$, we say that $P$ and $Q$ satisfy the circular interlacing condition if between any two roots of $P$, there is a root of $Q$ on the unit circle, and likewise, for any pair of roots of $Q$. It is easy to see that if $P$ and $Q$ satisfy the circular interlacing condition, then the degrees of both polynomials have to be equal, they cannot have repeated roots, and $x^2-1$ must divide $PQ$.

	A less obvious observation is that satisfying circular interlacing condition is preserved
	under addition of pairs, which is defined by considering each pair as a quotient.
	That is, given $Q_1,P_1$ and $Q_2,P_2$ satisfying the circular interlacing condition,
	by adding the respective fractions we get $\frac{Q_1}{P_1}+\frac{Q_2}{P_2}=\frac{Q}{P}$,
	where $P$ and $Q$ preserve the circular interlacing condition~\cite[Prop.~6]{MS3}.

	All the possible pairs (up to minor constraints) of polynomials satisfying the
	circular interlacing condition were classified in~\cite{MS2}.
	We restrict ourselves to a family given by
	\begin{equation}\label{def:pq}
	P(x)=\frac{(x^{m+n}-1)}{(x-1)}
	\quad\text{and}\quad
	Q(x)=\frac{(x^m-1)(x^n-1)}{(x-1)},
	\end{equation}
	where $\gcd(m,n)=1$. Our interest in pairs of this type
	derives from \cite[Prop.~5(a)]{MS3}, where it was shown that for
	$P$ and $Q$ as in \eqref{def:pq}, $(x^2-1)P(x)-xQ(x)$ is
	the minimal polynomial of a Salem number,
	possibly multiplied by a cyclotomic polynomial. In particular,
	if we take pairwise coprime integers $p_1,\dots,p_7,n\geq 2$ and define
	\[
	\frac{Q(x)}{P(x)}=
	\frac{x^{p_1+p_2}-1}{(x^{p_1}-1)(x^{p_2}-1)}
	+
	\frac{x^{p_3+p_4}-1}{(x^{p_3}-1)(x^{p_4}-1)}
	+
	\frac{x^{p_5+p_6}-1}{(x^{p_5}-1)(x^{p_6}-1)}
	+
	\frac{x^{p_7+n}-1}{(x^{p_7}-1)(x^n-1)},
	\]
	with $P$ and $Q$ relatively prime, and
	\begin{equation}\label{def:sevenprimes}
	f(x) = (x^2-1)P(x) - x Q(x),
	\end{equation}
	we have the following version of Lemma~5 from~\cite{MY}:
	
	\begin{lemma}\label{2009:lemma}
		Given pairwise coprime integers $p_1,\ldots,p_7,n\ge 2$,
		then the polynomial $f$ defined as in \eqref{def:sevenprimes} is either:
		\begin{itemize}
			\item the minimal polynomial of a Salem number of trace $-3$ and degree $n + p_1 + \cdots + p_7- 5$.
			\item a product of the minimal polynomial of a Salem number and a cyclotomic polynomial.
			\item a product of a quadratic polynomial with a root bounded by one, and a cyclotomic polynomial. 
		\end{itemize}
	\end{lemma}

	Lemma \ref{2009:lemma} does not require $p_1,\dots,p_7$ to be prime;
	however, this is a convenient choice and we will use it.
	Once these numbers are fixed, as $n$ varies we obtain an infinite family of polynomials.
	Below we will provide a list of finitely many families,
	corresponding to specific seven-tuples of primes $p_1,\ldots,p_7$,
	such that for all $n\geq 2$ coprime to $q=p_1\cdots p_7$,
	the polynomial $f$ defined in \eqref{def:sevenprimes} is always irreducible---hence Salem.
	Each family will give Salem numbers of trace $-3$ and sufficiently large even degrees that avoid
	$q-\varphi(q)$ residue classes modulo $q$.
	By using several families we make sure that every sufficiently large even degree is attained.

	\begin{prop}\label{prop:tuples}
		For each of the 21 seven-tuples of primes $(p_1,\ldots,p_7)$
		in Table \ref{table-tuples} and each integer $n\geq 2$ coprime to $p_1\cdots p_7$,
		the polynomial $f$ defined in \eqref{def:sevenprimes}
		is irreducible and Salem.
		Moreover, we obtain in this way Salem numbers of degree $72,76,78$ and every even degree $\geq 82$.
		Finally, the collection of seven-tuples in Table \ref{table-tuples} is minimal
		in the sense that any collection of tuples of the form $(2,3,p_3,\dots,p_7)$ with $p_i<29$
		misses infinitely many even degrees.
	\end{prop}
	\begin{table}[!ht]
		\begin{center}\vspace{-31pt}
			\begin{tabular}{p{0.9\textwidth}}
				\begin{multicols}{3}
					\begin{enumerate}[1.]
						\item (2, 3, 5, 7, 11, 13, 17)
						\item (2, 3, 5, 7, 11, 13, 19)
						\item (2, 3, 5, 7, 11, 13, 23)
						\item (2, 3, 5, 7, 11, 13, 29)
						\item (2, 3, 5, 7, 11, 17, 23)
						\item (2, 3, 5, 7, 11, 17, 29)
						\item (2, 3, 5, 7, 13, 17, 19)
						\item (2, 3, 5, 7, 13, 17, 29)
						\item (2, 3, 5, 7, 13, 19, 23)
						\item (2, 3, 5, 7, 13, 19, 29)
						\item (2, 3, 5, 7, 13, 23, 29)
						\item (2, 3, 5, 7, 17, 19, 29)
						\item (2, 3, 5, 11, 13, 17, 19)
						\item (2, 3, 5, 11, 13, 23, 29)
						\item (2, 3, 5, 11, 17, 19, 23)
						\item (2, 3, 5, 11, 19, 23, 29)
						\item (2, 3, 5, 13, 17, 19, 23)
						\item (2, 3, 7, 11, 13, 17, 19)
						\item (2, 3, 7, 11, 13, 19, 29)
						\item (2, 3, 7, 13, 17, 19, 23)
						\item (2, 3, 11, 17, 19, 23, 29)
					\end{enumerate}
				\end{multicols}
			\end{tabular}
		\end{center}
		\caption{Seven-tuples of primes used to construct infinite families of Salem numbers in Proposition~\ref{prop:tuples}.}\label{table-tuples}
	\end{table}
	\begin{proof}
		The proof goes along the same lines as Lemma 6 in \cite{MY}.
		In view of Lemma \ref{2009:lemma}, we need to rule out the possibility that
		$f$ has cyclotomic factors. To detect this, we introduce a new variable $y$
		and consider
		\[
		\frac{Q(y,x)}{P(y,x)}=
		\frac{x^{p_1+p_2}-1}{(x^{p_1}-1)(x^{p_2}-1)}
		+
		\frac{x^{p_3+p_4}-1}{(x^{p_3}-1)(x^{p_4}-1)}
		+
		\frac{x^{p_5+p_6}-1}{(x^{p_5}-1)(x^{p_6}-1)}
		+
		\frac{yx^{p_7}-1}{(x^{p_7}-1)(y-1)}.
		\]
		Define the curve $C(y,x):(x^2-1)P(y,x)-xQ(y,x)=0$,
		so that $C(x^n,x)$ gives us~(\ref{def:sevenprimes}).
		By~\cite[Lemma 1]{BS}, if $(y,x)$ is a cyclotomic point on $C$
		then so is either $(-y,-x)$, $(y^2,x^2)$ or $(-y^2,-x^2)$.
		Letting $C_1(y,x)=C(-y,-x)$, $C_2(y,x)=C(y^2,x^2)$ and $C_3(y,x)=C(-y^2,-x^2)$,
		it remains to confirm that for each tuple of primes from Table~\ref{table-tuples},
		we indeed do not have cyclotomic points. For each of the three pairs
		$(C, C_1)$, $(C,C_2)$, and $(C,C_3)$ we eliminate $x$ or $y$
		to obtain a single-variable polynomial. There are the following cases that we get:
		\begin{enumerate}[I.]
			\item After the elimination of either variable, we obtain an irreducible polynomial. 
			\item For $(p_1,\ldots,p_7)$, after elimination we find factors $\Phi_{p_i}$, which are permitted given that $\Phi_{p_i}$ divides $P(x),$ and $\gcd(P(x),Q(x))=1.$   
			\item For $(2,3,p_3,\ldots,p_7)$, after the elimination of $y$ we find a factor $\Phi_{12}$ and the elimination of $x$, a factor of $\Phi_4.$ If $(y, x)$ is cyclotomic point on both curves, then $y = x^n$ and $\exp( 2\pi i/4 ) = \exp ( 2\pi ikn/ 12 )$. Given that $\gcd(k, 12) = 1$, indicates  that $3$ divides $n$, which goes against the choice of $n$.
			\item For $(2,3,5,p_4,\ldots,p_7)$, after the elimination of $y$ we find a factor $\Phi_{20}$, and the elimination of $x$ is divisible by $\Phi_4.$ Similar to the case above, we conclude that $5$ divides $n$, which would contradict the choice of $n$.
		\end{enumerate}
		Since these are the only cases we find for all our 7-tuples—computed in PARI/G--we get these infinite families of Salem numbers of trace $-3$.
		
		To verify that all sufficiently large even degrees are obtained,
		we show that every residue class modulo $q=2\times 3\times\cdots\times 29$
		is obtained by selecting at least one of the tuples and one value of $n$
		of those allowed. Since $n+q$ is also allowed, we obtain the claim.
		The last part of the proposition follows from a similar check,
		this time showing that at least one
		residue class cannot be obtained if one uses only primes up to $23$.
		For more details, see~\cite[Lemma~6]{MY} and the PARI/GP code we used for the
		numerical verification of our statements~\cite{code}.
		The smallest degree Salem numbers we can construct as such are the ones
		corresponding to the first two 7-tuples in the table, which are degrees~$72,74$ and~$76$.
	\end{proof}

	Note that the list we provide is also minimal in the sense that
	omitting any of the tuples will result in uncovered residues
	modulo $2\times 3\times\cdots\times 29$.

	\subsection{Quadruples of primes}\label{S2.2}

	To deal with the degrees $54,56,60,62,64,66,68,70,74$ and $80$
	we use an alternative construction.
	For each quadruple $(p_1,p_2,p_3,p_4)$ from the list
	\vspace{-14pt}
	\begin{equation}\label{def:quadruples}
	\begin{tabular}{p{0.725\textwidth}}
	\setlength{\columnsep}{0pt}
	\begin{multicols}{5}
	(7,11,13,17)
	(7,11,13,19)
	(7,13,17,19)
	(7,11,13,23)
	(7,11,17,23)
	(7,11,19,23)
	(7,13,19,23)
	(7,11,17,29)
	(7,13,19,29)
	(7,11,19,37),
	\end{multicols}
	\end{tabular}
	\end{equation}
	we start from the rational function
	\[
	\frac{Q(x)}{P(x)}=
	\frac{x^8 + x^7 - x^5 -x^4 - x^3 + x + 1}{x^8 + 2x^7 + 2x^6 + x^5 - x^3 - 2x^2 - 2x - 1}
	+
	\frac{x^{p_1+p_2}-1}{(x^{p_1}-1)(x^{p_2}-1)}
	+
	\frac{x^{p_3+p_4}-1}{(x^{p_3}-1)(x^{p_4}-1)}
	\]
	and, similarly to \eqref{def:sevenprimes}, define again
	\begin{equation*}
	f(x) = (x^2-1)P(x)-x Q(x).
	\end{equation*}
	Note that the two degree-eight polynomials appearing in the above construction
	have been classified as a sporadic pair of interlacing polynomials
	(see~\cite[Table~2, Label~P]{MS2}).

	\begin{lemma}\label{2007:lemma1}
		Let $(p_1,p_2,p_3,p_4)$ be one of the quadruples in \eqref{def:quadruples}
		and let $f$ be as in \eqref{def:sevenprimes}. Then $f$ is the minimal polynomial
		of a Salem number of trace $-3$ and degree $p_1+p_2+p_3+p_4+6$.
	\end{lemma}

	\begin{proof}
		The lemma can be verified directly in PARI/GP.
	\end{proof}
	We obtain the stated degrees $54,56,60,\dots,70$, $74$ and $80$.

	\subsection{Existing small-degree Salem polynomials}
	It is known that the smallest possible degree is $32$~\cite{WWW}
	and examples for degrees $34, 36, 38$ and $40$ have been announced in \cite{EMRS,ERS}.
	Polynomials with these degrees are listed at the end of the paper for completeness.
	We thank Jean-Marc Sac-\'{E}p\'{e}e for sharing them with us.

	\subsection{Numerical search for missing degrees}\label{S2.4}
	We explain how we constructed Salem polynomials of degree $2d=42,44,\dots,52$ and $58$.
	It suffices to find totally positive irreducible polynomials $S(x)\in \mathbb{Z}[x]$
	for each degree $d$, such that $x^dS(x+\frac{1}{x}+2)$ is the minimal polynomial
	of a Salem number and the trace of $S(x)$ is $2d-3$.
	In other words, $S(x)$ must be irreducible with only
	real roots, of which exactly one lies in the interval $(4,\infty)$
	and all the others are contained in $(0,4)$.
	
	We briefly describe the implementation of the search process while omitting the heuristic that underlies it. This heuristic originates in~\cite[\S 2.1--2.2]{M} and in essence, proposes looking for the desired polynomial among those having a fixed (and small) resultant with a finite set of auxiliary polynomials. The search process begins by constructing the auxiliary polynomials, which are chosen to have appropriate properties. Specifically, for a fixed degree $d$ and trace $2d-3$, we search through families of monic irreducible polynomials
	with all roots real and positive, whose degrees add up to $d-1$.
	Thus, consider pairwise coprime monic polynomials
	$Q_1,\ldots,Q_r\in \mathbb{Z}[x]$ with $\deg Q_1+\cdots+ \deg Q_r=d-1$.
	For each $Q_i$, select $P_i\in\mathbb{Z}[x]$ satisfying
	\[
	\deg P_i<\deg Q_i
	\quad\text{and}\quad
	\mathrm{Res}(Q_i,P_i)=\pm 1.
	\]
	Applying the Chinese Remainder Theorem, we find a polynomial $P\in \mathbb{Z}[x]$
	that satisfies $P\equiv P_i\pmod{Q_i}$; specifically, $P$ is the solution
	modulo $Q=\prod^r_{i=1}Q_i$ and has degree $<d-1$.
	If $t$ is the trace of $Q=x^{d-1}-tx^{d-2}+\cdots$, we let
	\[
	\widetilde{P}=(x-(2d-3)+t)Q+P.
	\]
	Its degree is $d$ and its trace is $2d-3$, and we have that
	$\widetilde{P}\equiv P_i\pmod{Q_i}$, so that the resultant with $Q_i$ is again $\pm 1$.
	It only remains to check whether $\widetilde{P}$ is irreducible, has all real roots
	and corresponds to a Salem polynomial. If it does, we have found the desired polynomial.
	If not, we repeat the process using different polynomials $P_i$.

	For the choice of $Q_i$ we opted for the ``real'' counterparts of cyclotomic polynomials, i.e. such that
	\[
	x^{\deg Q_i}Q_i(x+\frac{1}{x}+2)=\Phi_m(x),
	\]
	where $\Phi_m$ is the minimal polynomials of $m$th root of unity.
	With this definition, $Q_i$ has integer coefficients, is irreducible
	and has only positive real roots (it generates the maximal
	totally real subfield of the cyclotomic field given by $\Phi_m$).
	Following the discussion in \cite[\S 2.2]{M}, we observe that if $\mathrm{Res}(Q_i,P_i)=\pm 1$,
	then $P_i$ is associated to a unit in the corresponding number field; that is, if $\zeta$
	is a root of $Q_i$, then $P_i(\zeta)$ is a unit in $\mathbb{Z}[\zeta]$.
	Therefore, for a given positive integer $M$, we compute a system of fundamental units
	$\varepsilon_1,\dots,\varepsilon_{d_i-1}$ in $\mathbb{Z}[\zeta]$ (abbreviate $d_i=\deg Q_i$)
	and loop through all the possible products
	\[
	u\, \varepsilon_1^{a_1} \varepsilon_2^{a_2} \cdots \varepsilon_{d_i-1}^{a_{d_i-1}},
	\]
	with $u=\pm 1$ and $|a_j|<M$.
	Each such product provides a polynomial $P_i$ that can be used in the algorithm
	described in the previous paragraph.
	Salem polynomials of degree $42,44,\dots,52$ were found using such algorithm,
	for several initial choices of auxiliary polynomials $Q_1,\dots,Q_r$
	and sufficiently large input value $M$.


	\clearpage
	
	\appendix
	\section*{Appendix}
	We include in Table~\ref{appendix:table} the coefficients of the sporadic Salem polynomials
	with trace $-3$ and degree $34\leq 2d\leq 52$ and $58$, which were not found using the families in
	Sections~\ref{S2.1} and~\ref{S2.2} and needed therefore a dedicated numerical search.
	For each polynomial
	\[
	x^{2d}+a_{2d-1}x^{2d-1}+\cdots+a_1x+a_0,
	\]
	we list only the coefficients $(a_{2d},\dots,a_{d})$;
	the others are obtained by the symmetry $a_k=a_{2d-k}$.
	The polynomials with degree $34\leq 2d\leq 40$ were announced in~\cite{EMRS} and~\cite{ERS}, and can be found in~\cite{S}.
	The remaining ones, with degree $42\leq 2d\leq 52$ and $58$, are new
	and were found with the method described in Section~\ref{S2.4}. 
	\begin{table}[!ht]
		\footnotesize
		\renewcommand\arraystretch{1.2}
		\begin{tabular}{c|>{$}l<{$}}
			\hline
			degree & \text{Salem polynomials of trace $-3$ and degree from 34 to 52 and 58}\\
			\hline
			34 & 1,3,2,-10,-40,-89,-149,-208,-257,-293,-315,-322,-311,\\
			&\hfill -281,-237,-191,-156,-143\\
			\hline
			36 & 1,3,2,-10,-39,-82,-124,-146,-136,-96,-40,12,45,55,50,43,43,48,51\\
			\hline
			38 & 1,3,0,-20,-63,-118,-161,-176,-173,-183,-228,-298,\\
			&\hfill -357,-376,-360,-345,-364,-419,-479,-505\\
			\hline
			40 & 1,3,0,-24,-90,-212,-384,-580,-766,-913,-998,-999,\\
			&\hfill -891,-655,-293,163,658,1127,1510,1760,1847\\
			\hline
			42 & 1,3,1,-18,-72,-177,-340,-561,-840,-1182,-1593,-2072,\\
			&\hfill -2604,-3165,-3732,-4289,-4822,-5312,-5731,-6050,-6248,-6315\\
			\hline
			44 & 1,3,4,1,-9,-28,-56,-92,-133,-175,-214,-247,-271,-284,\\
			&\hfill -284,-270,-243,-207,-167,-129,-98,-78,-71\\
			\hline
			46 & 1,3,2,-10,-41,-94,-163,-235,-295,-331,-336,-311,-264,\\
			&\hfill -207,-152,-109,-85,-87,-120,-183,-265,-347,-407,-429\\
			\hline
			48 & 1,3,0,-21,-72,-156,-264,-378,-478,-550,-588,-595,-582,\\
			&\hfill -564,-553,-555,-568,-585,-598,-601,-590,-566,-535,-508,-497\\
			\hline
			50 & 1,3,1,-16,-57,-121,-197,-276,-361,-464,-591,-735,-882,-1024,-1164,\\
			&\hfill -1309,-1458,-1603,-1739,-1870,-2001,-2126,-2226,-2286,-2310,-2315\\
			\hline
			52 & 1,3,2,-11,-45,-104,-183,-272,-362,-450,-537,-626,-717,-805,-880,\\
			&\hfill -932,-957,-963,-967,-985,-1021,-1064,-1096,-1106,-1097,-1082,-1075\\
			\hline
			58 & 1,3,1,-17,-62,-134,-218,-296,-362,-427,-505,-599,-698,-792,-884,\\
			&\hfill -986,-1098,-1200,-1268,-1304,-1343,-1423,-1545,\\
			&\hfill -1669,-1752,-1787,-1804,-1832,-1868,-1885\\
			\hline
		\end{tabular}
		\captionsetup{font=small}
		\vspace{6pt}
		\caption{Coefficients of Salem polynomials with trace $-3$ and degree $34\leq 2d\leq 52$ and $58$.}\label{appendix:table}

	\end{table}
			\vspace{-.5cm}
	\setlength{\parskip}{-0pt}
	

\begin{thebibliography}{10}
		
		\bibitem{BS}
		F.~Beukers and C.~J. Smyth,
		\emph{Cyclotomic points on curves},
		Number theory for the millennium {I}, A~K~Peters, Natick, MA, 2002, 67--85.

        \bibitem{CW}
        Q.~Chen and Q.~Wu,
        \emph{Salem numbers with minimal trace},
        Math.~Comp.~\textbf{92} (2023), no.~342, 1779--1790. 

		\bibitem{code}
		G.~Cherubini and P.~Yatsyna,
		\emph{Ancillary code},
		\url{https://github.com/pav10/salem-3}.
		
		\bibitem{EMRS}
		S.~El~Otmani, A.~Maul, G.~Rhin, and J.-M. Sac-\'{E}p\'{e}e,
		\emph{Finding degree-$16$ monic irreducible integer polynomials of minimal
			trace by optimization methods},
		Exp.~Math.~\textbf{23} (2014), no.~1, 1--5.
		
		\bibitem{ERS}
		S.~El~Otmani, G.~Rhin, and J.-M. Sac-\'{E}p\'{e}e,
		\emph{A {S}alem number with degree $34$ and trace {$-3$}},
		J.~Number Theory \textbf{150} (2015), 21--25.
		
		\bibitem{M}
		J.~McKee,
		\emph{Computing totally positive algebraic integers of small trace},
		Math.~Comp.~\textbf{80} (2011) no.~274, 1041--1052.
		
		\bibitem{MS1}
		J.~McKee and C.~Smyth,
		\emph{Salem numbers of trace {$-2$} and traces of totally positive
			algebraic integers},
		Algorithmic number theory, Lecture Notes in Comput.~Sci.~\textbf{3076} Springer, Berlin, 2004, 327--337.
		
		\bibitem{MS3}
		J.~McKee and C.~Smyth,
		\emph{There are {S}alem numbers of every trace},
		Bull.~London Math.~Soc.~\textbf{37} (2005) no.~1, 25--36.
		
		\bibitem{MS2}
		J.~McKee and C.~Smyth,
		\emph{Single polynomials that correspond to pairs of cyclotomic polynomials
			with interlacing zeros},
		Cent.~Eur.~J.~Math.~\textbf{11} (2013) no.~5, 882--899.
		
		\bibitem{MY}
		J.~McKee and P.~Yatsyna,
		\emph{Salem numbers of trace {$-2$}, and a conjecture of {E}stes and
			{G}uralnick},
		J.~Number Theory \textbf{160} (2016), 409--417.

        \bibitem{S} J.-M. Sac-\'{E}p\'{e}e, \emph{Small degree Salem numbers with trace -3}, preprint,
\href{https://hal.science/hal-04482471/file/salemList.pdf}{hal-0448247}.
		
		\bibitem{S1}
		C.~J.~Smyth,
		\emph{Salem numbers of negative trace},
		Math.~Comp.~\textbf{69} (2000) no.~230, 827--838.
		
		\bibitem{S2}
		C.~Smyth,
		\emph{Seventy years of {S}alem numbers},
		Bull.~Lond.~Math.~Soc.~\textbf{47} (2015) no.~3, 379--395.
		
		\bibitem{WWW}
		C.~Wang, J.~Wu and Q.~Wu,
		\emph{Totally positive algebraic integers with small trace},
		Math.~Comp.~\textbf{90} (2021) no.~331, 2317--2332.
		
	\end{thebibliography}
\end{document}